\documentclass[12pt]{article}
\usepackage{amsmath,amssymb,amsthm, graphicx} 

\newtheorem{theorem}{Theorem}
\newtheorem{lemma}[theorem]{Lemma}
\newtheorem{corollary}[theorem]{Corollary}
\newtheorem{proposition}{Proposition}
\newtheorem{problem}{Problem}

\title{Monk Algebras and Representability}
\author{Jeremy F.~Alm}
\date{January 2025}

\begin{document}

\maketitle

\begin{abstract}
    In ``Monk Algebras and Ramsey
Theory,'' \emph{J. Log. Algebr. Methods Program.} (2022), Kramer and Maddux prove various representability results in furtherance of the goal of finding the smallest weakly representable but not representable relation algebra. They also pose many open problems.  

In the present paper, we address problems and issues raised by Kramer and Maddux. In particular, we prove that their Proposition 7 does not generalize, and we answer Problem 1.1 in the negative: relation algebra $1311_{1316}$ is not representable. Thus $1311_{1316}$ is a good candidate for the smallest weakly representable but not representable relation algebra. 

Finally, we give the first known finite cyclic group representations for relation algebras $31_{37}$, $32_{65}$,  $1306_{1314}$, and $1314_{1316}$.  
\end{abstract}

\section{Introduction}

In ``Monk Algebras and Ramsey Theory'' \cite{KramerMaddux}, Kramer and Maddux  seek to find the smallest weakly representable but not representable relation algebra. (The smallest known such algebra has 7 atoms, and is given in \cite{AlmHirschMaddux}.)  They consider algebras derived from ``color algebras'' by splitting atoms (in the sense of \cite{AndMadd}).  Let $\mathfrak{E}_n(\{2,3\})$ denote the finite integral symmetric relation algebras with $n$ atoms, where the forbidden diversity cycles are exactly the 1-cycles, i.e., cycles of the form $aaa$, where $a$ is a diversity atom.  Kramer and Maddux call these ``color algebras.''  Given  $\mathfrak{E}_n(\{2,3\})$ and a diversity atom $a$, let

\[
      \mathfrak{E}_n(\{2,3\})\left(\frac{a}{a \ \ a'}\right)
\]
denote the algebra derived from  $\mathfrak{E}_n(\{2,3\})$ by splitting the atom $a$ into two symmetric atoms $a$ and $a'$, and let

\[
      \mathfrak{E}_n(\{2,3\})\left(\frac{a}{r \ \ \breve{r}}\right)
\]
denote the algebra derived from  $\mathfrak{E}_n(\{2,3\})$ by splitting the atom $a$ into an asymmetric pair of atoms $r$ and $\breve{r}$.  (See the discussion in Section 6 of \cite{KramerMaddux}.) We will use this notation to connect our work to that of Kramer and Maddux, but we will also define the algebras under consideration directly via their atoms and diversity cycles.

In this paper, we will avail ourselves of the ``finite field method'' of constructing representations, due to Comer \cite{Comer83}.  Fix $m \in \mathbb{Z}^{+}$.  We search over primes $p \equiv 1 \pmod{m}$, where a desirable $p$ satisfies the following. Let $X_{0} := H \leq \mathbb{F}_{p}^{\times}$ be the unique subgroup of order $(p-1)/m$. Now let $X_{1}, \ldots, X_{m-1}$ be the cosets of $\mathbb{F}_{p}^{\times}/X_{0}$. In particular, as $\mathbb{F}_{p}^{\times}$ is cyclic, we may write $X_{i} = g^{i}X_{0} = \{ g^{am+i} : a \in \mathbb{Z}^{+} \},$ where $g$ is a generator of $\mathbb{F}_{p}^{\times}$. 

For each $0 \leq i \leq m-1$, define $R_{i} := \{(x,y) \in \mathbb{F}_{p} \times \mathbb{F}_{p} : x - y \in X_{i} \}.$ Here, the sets $R_{0}, \ldots, R_{m-1}$ together with $\text{Id} = \{ (u, u) : u \in \mathbb{F}_{p} \}$ are the atoms in our relation algebra.  As noticed by Comer, this construction always gives a relation algebra. 

The diversity atoms are either all symmetric or all asymmetric, as follows:

\begin{lemma}
  For all $0 \leq i \leq m-1$, we have that $X_{i} = -X_{i}$  if and only if $(p-1)/m$ is even. 
\end{lemma}

Notice that $(p-1)/m)$ is the order of each of the cosets, and that $X_{i} = -X_{i}$ if and only if $R_i = R_i^{-1}$. 

The following symmetry property turns out to be important:

\begin{lemma}
    For each $i,j,k < m$, 
    
\[
 X_i + X_j \supseteq X_k \text{ \ if and only if \ }X_0 + X_{j-i} \supseteq X_{k-i}
\]    
(where indices are computed modulo $m$).
\end{lemma}

This makes checking which cycles are mandatory much more manageable. Furthermore, the algorithm given in \cite{AlmY} makes this checking extremely fast. 

Comer told the author in an email communication dated 3 October 2016 that the calculations in \cite{Comer83} were done by hand.  In a talk given at the 2011 Spring Central Sectional Meeting of the AMS at the University of Iowa, Roger Maddux mentioned that he found representations for $\mathfrak{E}_{n}(\{2,3\})$ for $n = 7$ and $n = 8$ (six and seven colors, respectively) using Comer's technique but on a modern computer.  Maddux did not find a representation for the 8-color algebra $\mathfrak{E}_{9}(\{2,3\})$. 

In \cite{AlmManske}, Manske and the author used the same technique to find representations for all numbers of colors less than 400, except for 8 colors and 13 colors.  A representation for the 8-color algebra using Comer's method was ruled out by checking primes up to the Ramsey bound $R(3,3,3,3,3,3,3,3)$.  (Note that this Ramsey bound grows at least exponentially in the number of colors \cite{Chung}.) Around the same time, Kowalski also found representations for the $n$-color algebras for $n\leq 120$, except for 8 and 13, using Comer's technique, included some non-prime finite fields \cite{Kowalski}. 

Using the  fast algorithm  in \cite{AlmY}, the author found representations for the $n$-color algebras for $401 \leq n \leq 2000$, and ruled out a Comer representation for 13 colors by reducing the exponential Ramsey bound to a polynomial bound for Comer representations using a Fourier-analytic technique from additive number theory \cite{Alm401}. In particular, the $n$-color algebra is not representable via Comer's finite field method on a prime field of order greater than $n^4+5$. (Note that this bound does not apply to non-prime finite fields.)

\section{Examples}

Let's recall Proposition 7 from \cite{KramerMaddux}, as follows.

\begin{proposition}[\cite{KramerMaddux} Proposition 7]
    If $\mathfrak{A}$ is obtained from $\mathfrak{E}_{5}(\{2,3\})$ by splitting all 4 diversity atoms into two parts each, then $\mathfrak{A}\not\in RRA$. Each atom may be split symmetrically or asymmetrically. For example, 

    \[
        \mathfrak{E}_{5}(\{2,3\})\left(\frac{a}{a \ a'}\right)\left(\frac{b}{b \ b'}\right)\left(\frac{c}{c \ c'}\right)\left(\frac{d}{d \ d'}\right)\not\in RRA, 
    \]
but all 16 combinations produce non-representable relation algebras. 
    
\end{proposition}

The following two examples show that Proposition 7 from Kramer-Maddux \cite{KramerMaddux} does not generalize to  $\mathfrak{E}_{n}(\{2,3\})$ for all $n>5$. 

\begin{theorem}
    Proposition 7 from Kramer and Maddux \cite{KramerMaddux} does not generalize to $\mathfrak{E}_{n}(\{2,3\})$ for all $n$. In particular, we have the following two results:

    \begin{enumerate}

    \item[(a.)] Consider the 115-color Ramsey algebra $\mathfrak{E}_{116}(\{2,3\})$. Then $\mathfrak{E}_{116}(\{2,3\})\left(\frac{a_i}{a_i \ a_i'}\right)$, where each atom is split symmetrically, is representable.
        \item[(b.)] Consider the 31-color Ramsey algebra $\mathfrak{E}_{32}(\{2,3\})$. Then  $\mathfrak{E}_{32}(\{2,3\})\left(\frac{a_i}{a_i \ \breve{a_i}}\right)$, where each atom is split asymmetrically, is representable.

    \end{enumerate}
\end{theorem}

\begin{proof}

    To prove part (a.), let $p=751181$ and $m=115$. Define $X_0$,\ldots,$X_{114}$ as above. 

    To prove part (b.), let $p=33791$ and $m=31$. Define $X_0$,\ldots,$X_{61}$ as above. 
\end{proof}

The example used in the proof of part (a.) has another interesting property. Relation algebra $32_{65}$ was shown in \cite{AMM} to be representable over a finite set, namely a set of 416,714,805,914 points. This was reduced in \cite{DH} to 63,432,274,896 points, which was later reduced to 8192 by the first and fifth authors (unpublished), and finally to 3432 in \cite{MR3951643}.  The current ``record'' is a representation over 1024 points, found by the author \cite{Alm22}. Yet despite all this attention to $32_{65}$, no cyclic group representation has been given in the literature.  Until now.

\begin{corollary}
    Relation algebra $32_{65}$ has a cyclic group representation over $\mathbb{Z}/p\mathbb{Z}$ for $p=751181$.
\end{corollary}

\begin{proof}
    As above, define $X_0$,\ldots,$X_{114}$. Then let 

    \begin{align*}
    \rho(a)&=\cup_{i=2}^{114} X_i\\
    \rho(b)&= X_0\\
    \rho(c)&=X_1\\
\end{align*} is a group representation. 
\end{proof}

Similarly, the example in the proof of part (b.) can be used to show that relation algebras $31_{37}$ and $1306_{1314}$ are finitely representable. 

\begin{corollary}
    Relation algebras $31_{37}$ and $1306_{1314}$ have cyclic group representations over $\mathbb{Z}/p\mathbb{Z}$ for $p=33791$. 
\end{corollary}

\begin{proof}
    As above, define $X_0$,\ldots,$X_{61}$.  For $31_{37}$, define 

    \begin{align*}
    \rho(a)&=\bigcup_{i=1}^{30} X_i \ \cup \ \bigcup_{i=32}^{61} X_i\\
    \rho(r)&=X_0\\
    \rho(\breve{r}) &= X_{31}
\end{align*} is a group representation.

  For $1306_{1314}$, define 
    \begin{align*}
    \rho(a)&=\bigcup_{i=2}^{30} X_i \ \cup \ \bigcup_{i=33}^{61} X_i\\
    \rho(b)&= X_1\cup X_{32}\\
    \rho(r)&=X_0\\
    \rho(\breve{r}) &= X_{31}
\end{align*} is a group representation. 
    
    \end{proof}

---------------------------------------------------------------------------

\section{$1308_{1316}$}

Relation algebra $1308_{1316}$ has atoms $1'$, $a$, $b$, $r$, and $\breve{r}$. The forbidden cycles are $aaa$, $bbb$, and $rr\breve{r}$. This algebra is obtained from splitting a diversity atom in $\mathfrak{E}_{4}(\{2,3\})$ and then ``adding back'' the mandatory cycle $rr\breve{r}$:

    \[
        1308_{1316} = \mathfrak{E}_{4}(\{2,3\})\left(\frac{c}{r \ \breve{r}}\right) + rrr. 
    \]
    Problem 1.3 from \cite{KramerMaddux} asks if $1308_{1316}$ is representable. While we have not been able to answer this question, we can provide some information on a subalgebra of $1308_{1316}$.

 Algebra $33_{37}$ is a subalgebra of $1308_{1316}$ formed by replacing atoms $a$ and $b$ with $a+b$; $a+b$ is a flexible atom.  Many representations of $33_{37}$ have been found, and it is probably representable on arbitrarily large finite sets.  Because it has a flexible atom, it is also representable over a countable set. 

Since a representation of $1308_{1316}$ over a set $U$ implies a representation of $33_{37}$ over $U$ as well, we will take a minute to study representations of $33_{37}$. The following result seems interesting in its own right. 

\begin{theorem}
    The set of integers $n\leq 100$ such that $33_{37}$ is representable over $\mathbb{Z}/n\mathbb{Z}$ is $\{29, 38, 39, 41\} \cup [43,  100]$.  
\end{theorem}

For example, here is a representation over $\mathbb{Z}/38\mathbb{Z}$
\begin{align*}
    A &= \{1, 5, 6, 8, 9, 14, 16, 17, 18, 19, 20, 21, 22, 24, 29, 30, 32, 33, 37\} \\
    R &= \{ 3, 7, 10, 11, 13, 23, 26, 34, 36 \}\\
    R^{-1} &= \{ 2, 4, 12, 15, 25, 27, 28, 31, 35 \} 
\end{align*}

\begin{proof}
    The representation over $\mathbb{Z}/29\mathbb{Z}$ is presented in \cite{Directed}, and is probably of minimal size.  The find the other representations, and rule out representations for certain $n$, we used the SAT solver CryptoMiniSAT via the python library \texttt{boolexpr}.   

 To that end, we consider the finite cyclic group of order $n$, namely $\mathbb{Z}/n\mathbb{Z}$. For each $x\in (\mathbb{Z}/n\mathbb{Z})^\times$, define four boolean variables:
\begin{itemize}
    \item $\phi_{x,0}$ means that $x$ is labelled with atom $a$
    \item $\phi_{x,1}$ means that $x$ is labelled with atom $r$
    \item $\phi_{x,2}$ means that $x$ is labelled with atom $\breve{r}$
\end{itemize}

We will build a boolean formula $\Phi$ that declares that any satisfying assignment corresponds to a graph coloring that yields a representation.  We will build several sub-formulae, whose conjunction will be the definition of $\Phi$. 

First, we will declare that, for each fixed $x$, exactly one of the $\phi_{x,k}$ is TRUE:

\[
    \Phi_0 = \bigwedge_{x \neq 0} \left [\bigvee_{k=0}^3 \phi_{x,k} \ \wedge \  \bigwedge_{k_\alpha\neq k_\beta} ( \neg\phi_{x,k_\alpha} \vee \neg\phi_{x,k_\beta}) \right ]
\]

Next we say that $a$ is symmetric, while $r$ and $\breve{r}$ are converses. (While it might be natural to use biconditionals here, we will express the formula in parallel to the particular implementation in \texttt{boolexpr}, which inexplicably lacks a built-in biconditional function.) 

\begin{align*}
    \Phi_1 = \bigwedge_{x\neq 0}  [ &(\phi_{x,0} \rightarrow \phi_{-x,0}) \\
                               \wedge \ &(\phi_{x,1} \rightarrow \phi_{-x,3}) \\
                               \wedge \ &(\phi_{x,2} \rightarrow \phi_{-x,2})  ] \\
\end{align*}

Next we declare that all ``needs'' are met. Define the set ALLNEEDS$= \{0,1,2\} \times \{0,1,2\}$. We will define the sets of needs of each type of edge by subtracting forbidden triangles from ALLNEEDS.

So let 
\begin{align*}
     A &= \mathrm{ALLNEEDS}  \\
     R &= \mathrm{ALLNEEDS} \setminus \{(2,2)\} \\
     Rconv &= \mathrm{ALLNEEDS} \setminus \{(1,1)\} \\
    \end{align*}

    Define the formula $\psi_{x,c_1,c_2}$, which declares that  $x$ has its $(c_1,c_2)$ need met, as follows:

    \[
        \psi_{x,c_1,c_2} = \bigvee_{\substack{y,z \neq 0 \\ x=y+z}} \phi_{y,c_1} \wedge \phi_{z,c_2}
    \]

Now we define a formula asserting that all $x$ have all needs met:

\begin{align*}
    \Phi_2 &= \bigwedge_{x\neq 0} \left [\phi_{x,0} \rightarrow \bigwedge_{(c_1, c_2)\in A} \psi_{x,c_1,c_2}\right ] \\
    &\wedge \bigwedge_{x\neq 0} \left [\phi_{x,1} \rightarrow \bigwedge_{(c_1, c_2)\in R} \psi_{x,c_1,c_2}\right ] \\
    &\wedge \bigwedge_{x\neq 0} \left [\phi_{x,2} \rightarrow \bigwedge_{(c_1, c_2)\in Rconv} \psi_{x,c_1,c_2}\right ] \\
\end{align*}

Finally, we define a formula forbidding the forbidden triangles.  Let

\[
    \Phi_3 = \bigwedge_{y+z\neq 0} \left [ \neg (\phi_{y,1} \wedge \phi_{z,1} \wedge \phi_{y+z,2})  \wedge  \neg (\phi_{y,2} \wedge \phi_{z,2} \wedge \phi_{y+z,1}) \right ]
\]

Finally, define $\Phi = \Phi_0 \wedge \Phi_1 \wedge \Phi_2 \wedge \Phi_3 $.  Using the solver CryptoMiniSAT, we verified  that for $1 \leq n \leq 100$,  $\Phi$ is satisfiable for $n \in \{29, 38, 39, 41\} \cup [43,  100]$.

\end{proof}

\section{$1311_{1316}$}
Relation algebra $1311_{1316}$ has atoms $1'$, $a$, $b$, $r$, and $\breve{r}$. The forbidden cycles are $aaa$, $bbb$, and $rrr$.   Because any directed $K_4$ must contain a chain, the number of points in any representation  is strictly bounded above by $R(3,3,4) = 30$ \cite{R433}. Hence we may consider only $n\leq 29$.  

\begin{lemma}
    Relation algebra $1311_{1316}$ is not representable over 28 points, nor over 29 points. 
\end{lemma}

\begin{proof}
    Consider a supposed representation on $K_{n}$, $n\leq 29$.   Suppose any vertex has $a$-degree at least nine. The graph induced by the $a$-neighborhood of that vertex will either contain an $a$-edge, creating a monochromatic $a$-triangle, or the neighborhood is colored in $b$ and $r$ only. In this latter case, since $R(3,4) = 9$, we get either a  monochromatic $b$-triangle or a  monochromatic $r$-$K_4$, which contains a chain.  Hence we see that the $a$-degree of any vertex is at most 8.  By symmetry, the $b$-degree of any vertex is at most 8 as well.

    Now consider the $r$-out-degree of any vertex $v$.  Notice that if there are any $r$-edges in graph induced by the $r$-out neighborhood of $v$, a chain is formed.  Thus that neighborhood contains only $a$- and $b$- edges. Since $R(3,3)=6$,  we can see that this neighborhood contains at most five vertices.  Thus the $r$-out-degree (and by symmetry, the $r$-in-degree) is at most 5. 

    For any vertex $v$, then, its total degree is at most $5+5+8+8=26$. Hence $n\leq 27$. 
\end{proof}

At this point, we will turn again to the mighty SAT solver. 

\begin{theorem}\label{bigthm}
    Relation algebra $1311_{1316}$ is not representable.
\end{theorem}

\begin{proof}
    We will use a SAT solver to rule our representations on $n$ points for all $n\leq 27$. To that end, we consider the complete directed graph on $n$ vertices, with a directed edge going both ways between vertices (since we have a pair of asymmetric atoms). Let our vertex set be $V = \{0, 1, \ldots, n-1\}$. For each directed edge $ij$ define four boolean variables:
\begin{itemize}
    \item $\phi_{i,j,0}$ means that $ij$ is labelled with atom $a$
    \item $\phi_{i,j,1}$ means that $ij$ is labelled with atom $b$
    \item $\phi_{i,j,2}$ means that $ij$ is labelled with atom $r$
    \item $\phi_{i,j,3}$ means that $ij$ is labelled with atom $\breve{r}$
\end{itemize}

We will build a boolean formula $\Phi$ that declares that any satisfying assignment corresponds to a graph coloring that yields a representation.  We will build several sub-formulae, whose conjunction will be the definition of $\Phi$. 

First, we will declare that, for each fixed $ij$, exactly one of the $\phi_{i,j,k}$ is TRUE:

\[
    \Phi_0 = \bigwedge_{i \neq j} \left [\bigvee_{k=0}^3 \phi_{i,j,k} \ \wedge \  \bigwedge_{k_\alpha\neq k_\beta} ( \neg\phi_{i,j,k_\alpha} \vee \neg\phi_{i,j,k_\beta}) \right ]
\]

Next we say that $a$ and $b$ are symmetric, while $r$ and $\breve{r}$ are converses. 

\begin{align*}
    \Phi_1 = \bigwedge_{i \neq j}  [ &(\phi_{i,j,0} \rightarrow \phi_{j,i,0}) \\
                               \wedge \ &(\phi_{i,j,1} \rightarrow \phi_{j,i,1}) \\
                               \wedge \ &(\phi_{i,j,2} \rightarrow \phi_{j,i,3}) \\
                               \wedge \ &(\phi_{i,j,3} \rightarrow \phi_{j,i,2})  ] \\
\end{align*}

Next we declare that all ``needs'' are met. Define the set ALLNEEDS$= \{0,1,2,3\} \times \{0,1,2,3\}$. We will define the sets of needs of each type of edge by subtracting forbidden triangles from ALLNEEDS.

So let 
\begin{align*}
     A &= \mathrm{ALLNEEDS} \setminus \{(0,0)\} \\
     B &= \mathrm{ALLNEEDS} \setminus \{(1,1)\} \\
     R &= \mathrm{ALLNEEDS} \setminus \{(2,2), (2,3), (3,2)\} \\
     Rconv &= \mathrm{ALLNEEDS} \setminus \{(2,3), (3,2), (3,3)\} \\
    \end{align*}

    Define the formula $\psi_{i,j,c_1,c_2}$, which declared that the edge $ij$ has its $(c_1,c_2)$ need met, as follows:

    \[
        \psi_{i,j,c_1,c_2} = \bigvee_{k \not\in\{i,j\}} \phi_{i,k,c_1} \wedge \phi_{k,j,c_2}
    \]

Now we define a formula asserting that all edges have all needs met:

\begin{align*}
    \Phi_2 &= \bigwedge_{i\neq j} \left [\phi_{i,j,0} \rightarrow \bigwedge_{(c_1, c_2)\in A} \psi_{i,j,c_1,c_2}\right ] \\
    &\wedge \bigwedge_{i\neq j} \left [\phi_{i,j,1} \rightarrow \bigwedge_{(c_1, c_2)\in B} \psi_{i,j,c_1,c_2}\right ] \\
    &\wedge \bigwedge_{i\neq j} \left [\phi_{i,j,2} \rightarrow \bigwedge_{(c_1, c_2)\in R} \psi_{i,j,c_1,c_2}\right ] \\
    &\wedge \bigwedge_{i\neq j} \left [\phi_{i,j,3} \rightarrow \bigwedge_{(c_1, c_2)\in Rconv} \psi_{i,j,c_1,c_2}\right ] \\
\end{align*}

Finally, we define a formula forbidding the forbidden triangles.  Let $F = \{(0,0,0), (1,1,1), (2,2,2), (2,3,2), (2,2,3), (3,3,2), (3,2,3), (3,3,3)\}$, where the triple $(x,y,z)$ means $x\cdot y;z = 0$. (Note that $(x,y,z)$ is not the same as the cycle $xyz$; this notation was chosen the make the formula and program easier to parse.) 

Then define
\[
    \Phi_3 = \bigwedge_{i < j < k} \left [\bigwedge_{(c_1, c_2,c_3) \in F} \neg (\phi_{i,j,c_1} \wedge \phi_{i,k,c_2} \wedge \phi_{k,j,c_3}) \right ]
\]

Finally, define $\Phi = \Phi_0 \wedge \Phi_1 \wedge \Phi_2 \wedge \Phi_3 $.  Using the solver CryptoMiniSAT, we verified the unsatisfiability of $\Phi$ for $n \leq 27$. 

\end{proof}

In order to verify that the conditions in the formulas in the previous proof are correct, it may be helpful to use Figure \ref{fig:1311}, where the red edges represent atom $a$, the blue edges atom $b$, and the forward-green edges $r$.  For example, the forbidden cycles in the third row are all equivalent to $rrr$, and correspond in order, left to right, to $(2,2,2), (2,2,3), (2,3,2) \in F$, respectively. (Trying to work ``up to isomorphism'' when coding the forbidden cycles is not worth the effort, and would make it harder to check correctness in any case. Each triangle $ijk$ is considered only once, with the ``bottom'' edge being $ij$, where  $i<j<k$.)   

\begin{figure}
    \centering
   \includegraphics[width=5in]{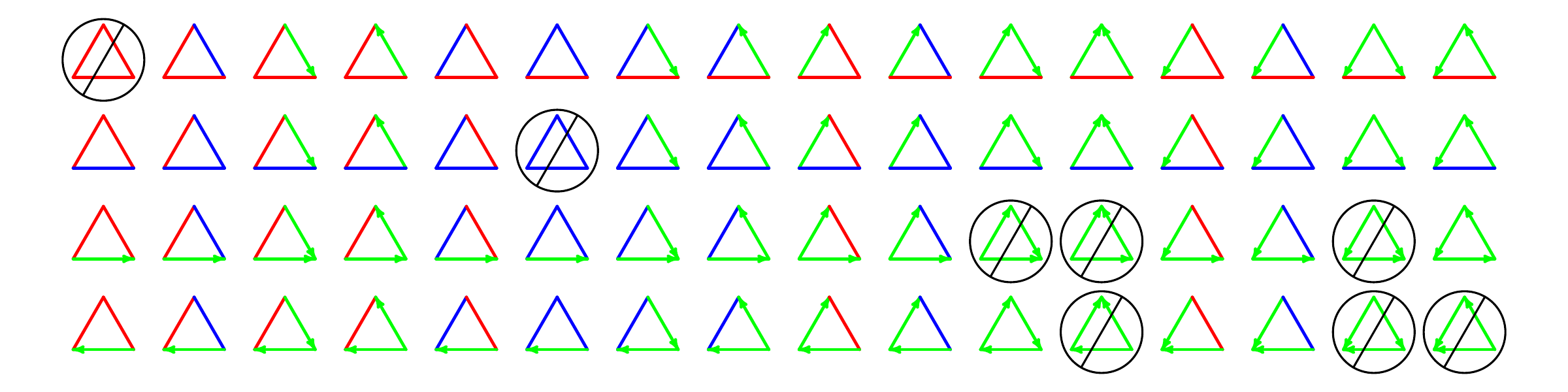}
    \caption{Cycle diagram for $1311_{1316}$}
    \label{fig:1311}
\end{figure}

\section{$1314_{1316}$}

Kramer and Maddux mention that $1314_{1316} = \mathfrak{E}_4(\{2,3\})\left(\frac{c}{r \ \breve{r}}\right) + rrr + rr\breve{r}$ is representable by a generalization of Comer's argument in \cite{Comer84}. Comer's proof produces representations over a countable set.  Here we give an explicit finite representation over 71 points.  

Let $p=71$, and consider the finite field $\mathbb{F}_p$.  Let $m = 10$, and define $X_0$,\ldots,$X_9$ as above. 

Then we define a group representation as follows:
\begin{align*}
    \rho(a) &= X_3 \cup X_8 \\
    \rho(b) &= X_4 \cup X_9 \\
    \rho(r) &= X_0 \cup X_1 \cup X_2 \\
    \rho(\breve{r}) &= X_5 \cup X_6 \cup X_7 
\end{align*}

\section{Summary and open questions}

We have shown in Theorem \ref{bigthm} that Problem 1.1 from \cite{KramerMaddux} has a negative answer: $1311_{1316}$ is not representable. The ``reason'' for non-representability is Ramsey-theoretic in nature, which gives this author hope that $1311_{1316}$ might be weakly representable, since Ramsey-theoretic results do not apply directly to weak representations in the same way that they do for representations.  (Weak representations can be thought of as partial edge-colorings of complete graphs, or equivalently of edge-colorings in which some edges may receive more than one color.) 

\begin{problem}
    Is $1311_{1316}$ weakly representable?
\end{problem}

We might also consider the odd case of the 8-color and 13-color algebras.  The author showed in \cite{AlmManske} and \cite{Alm401} that no finite field representation exist of these algebras. However, the reason for this seems to be that there just aren't enough primes below $8^4 +5$ and $13^4 + 5$, respectively, that are equivalent to $1\pmod{2\cdot8}$ and $1 \pmod{2\cdot 13}$, respectively. This really doesn't seem to give us a reason to believe that these algebras are not representable.  

\begin{problem}
    Is $\mathfrak{E}_9(\{2,3\})$ representable?
\end{problem}

On the other hand, for $n\geq 4$, there are no known representations of the color algebras that are NOT finite field representations.  

\begin{problem}
    Are there any representations of $\mathfrak{E}_n(\{2,3\})$ for $n>4$ that are not finite field representations? In particular, are there any representations over a number of points not equal to a prime or prime power? 
\end{problem}

Finally, the author's argument in \cite{Alm401} uses the so-called ``sum-product phenomenon'' in prime order fields; the bound of $n^4+5$ does not apply to prime-power fields, which do not exhibit this phenomenon.  

\begin{problem}
    Does there exist a representation of an $n$-color algebra over a field of order $p^k$, $k>1$, where $p^k > n^4+5$? If not, then there would have to be a different reason than the one given in \cite{Alm401}. 
\end{problem}

\section{Declarations}

\textbf{Ethical approval}

Not applicable. 

\textbf{Competing interests}

Not applicable. 

\textbf{Author's contribution}

The sole author performed this work on his own. 

\textbf{Availability of data and materials}

Not applicable. 

\textbf{Funding }

The author was not funded by any external agency.


\end{document}